\documentclass[11pt]{article}

\usepackage{amsmath,amssymb,amsthm}
  \usepackage{fullpage} 

\newtheorem{theorem}{Theorem}

\newtheorem{remark}{Remark}

\title{How to pick your team with no size restriction} 
\author{}
\date{}

\begin{document}
\author{
		Hiranya Kishore Dey \\ 
		Department of Mathematics,\\ 
        Indian Institute of Technology, Jammu.\\
        Jammu, 181221, Jammu and Kashmir, India. \\
        email: hiranya.dey@iitjammu.ac.in\\ \\     }

 \maketitle

\begin{abstract}
\medskip 
 
Settling a problem raised
by Eccles in 2015, Narayanan in 2026 considers 
a two-player game in which two captains alternately select players while the opponent decides to which team each selected player is assigned. 
Moreover, the two teams are required to
have equal cardinalities, and Narayanan proved that the second player has a non-losing strategy. 

In this paper, we study a natural variant in which the teams are allowed to have different cardinalities, and the winner is determined by 
comparing the average strengths of the two teams. We show that, in this setting, the parity of the total number of players completely determines which player has a non-losing strategy: the first player has a non-losing strategy when the number
of players is even, while the second player has a non-losing strategy when it is odd.
\end{abstract} 
	
	{\bf 2020 MSC}: 91A46, 91A05; 05A20

\section{Introduction}

 Eccles \cite{Eccles} communicated the following game to Narayanan as a puzzle.  

\medskip 
\noindent 
{\bf Game 1:} Suppose that there are $2m\ge2$ players whose strengths are represented by real
numbers
\[
a_1 \leq a_2 \leq \ldots \leq a_{2m}.
\]
 Alice and Bob act as captains and they must split them into two teams of $m$ players each. In every round, one captain acts as the picker by selecting an unassigned player, while the other serves as the chooser and decides which team that player joins.
 Alice 
 begins as the picker and Bob as the chooser, with the two roles alternating thereafter. The game continues until one captain has assembled a team of $m$ players, at which point all remaining players are automatically allocated to the other team. The captain whose team has the greater total score is declared the winner.
 
  Narayanan \cite{Narayanan} proved the following. 
\begin{theorem}
\label{thm:Narayanan}
Alice cannot win Game 1 on any board of even cardinality.
\end{theorem}

We consider the following variant of Game 1, in which teams are allowed to have different cardinalities and the winner is determined by average strength rather than total strength.
% In this paper, we consider the following variant of Game 1. Unlike Game 1, we allow teams to have different sizes and the winner is determined by the average strength of the resulting teams. Replacing the total strength by the average strength allows teams of unequal cardinality and leads to a qualitatively different optimization criterion. This variant is motivated by scenarios in which a participant may be assigned multiple tasks or matches, making the average strength of the selected team more relevant than its total strength.
More precisely, consider the following. 

\medskip 
\noindent 
{\bf Game 2:} Suppose that there are $n\ge2$ players whose strengths are represented by real
numbers
\[
a_1 \leq a_2 \leq \ldots \leq a_n.
\]
Initially both teams are empty.
Alice first selects one unchosen player, and Bob decides to which team that
player is assigned. Next Bob selects one of the remaining players and Alice
decides to which team that player is assigned. The players continue
alternating in 
this fashion until every player has been assigned.
If one team is empty, then by convention the other team is declared the winner.
Otherwise, the team with the larger average strength wins. Equality of averages
is regarded as a draw.

Let $G(a_1,a_2,\ldots,a_n)$ denote an instance of Game~2 in which the players have strengths \linebreak $(a_1,a_2,\ldots,a_n).$ 
Our main result is the following.

\begin{theorem}
\label{thm:main}
Consider the game $G(a_1,a_2,\ldots,a_n)$ and let 
\[
x = \frac{a_1 + \cdots + a_n}{n}
\]
denote their average.

\begin{enumerate}
\item If $n$ is even, then Alice has a strategy ensuring that she does not lose the game.
\item If $n$ is odd, then Bob has a strategy ensuring that Alice does not win the game.
\end{enumerate}
Moreover, 
if the multiset of values with odd multiplicity is not symmetric about $x$, we have the following.
\begin{enumerate}
\item If $n$ is even, Alice has a strategy ensuring her win.
\item If $n$ is odd, Bob has a strategy ensuring his win.
\end{enumerate}
\end{theorem}

Theorem \ref{thm:main} stands in sharp contrast to Theorem \ref{thm:Narayanan}. In Game 1, where the two teams are required to have the same cardinality, the second player has a non-losing strategy. In our variant, removing this balance condition fundamentally changes the outcome of the game. 
% the player who moves first has a non-losing
% strategy when the number of players is even, while the second player has a
% non-losing strategy when the number of players is odd.

\section{Proof of Theorem \ref{thm:main}}
\label{sec:proof-thm}

The main challenge in the proof lies in identifying the correct opening move for the first player. 
Indeed, we construct examples (see Remark~\ref{rem:examples}) showing that if Alice deviates from a particular first move, then Bob can force a win under optimal play. 
% We first consider the case when $n$ is even say $n=2m$ and prove the following. 

% \begin{theorem}
% Let $2m$ players have strengths
% \[
% a_1\le a_2\le\cdots\le a_{2m},
% \]
% and let
% \[
% x=\frac1{2m}\sum_{i=1}^{2m}a_i
% \]
% be their average strength.

% Then Alice has a strategy which guarantees that her average strength is at least $x$. Moreover, if the multiset of values occuring an odd number of times is not symmetric about $x$, Alice has a strategy that guarantees that her average strength is more than x. 
% \end{theorem}

\begin{proof}
If $n=2$ or $3$, the result follows by straightforward case analysis. So, we assume $n \geq 4$. 
We first the consider the case when $n$ is even, say $n=2m$. 

\medskip

\noindent 
{\bf \large When $n$ is even:} 

\medskip 
\noindent
We subtract $x$ from every player's strength. This does not affect the game, since both Alice's average and the overall average decrease by the same amount. Hence we may assume throughout that
\[
\sum_{i=1}^{2m}a_i=0.
\]

We now describe Alice's strategy. We partition the game into m stages, each consisting of two turns. In the first turn of a stage, Alice picks a player and Bob assigns that player to a team. In the second turn, Bob picks a player and Alice assigns that player to a team.

At the beginning of each stage, some players remain unassigned. During the stage, exactly two of these players are permanently assigned, one in each turn according to the rules of the game. These two players are then removed from further consideration. Consequently, after the k-th stage, exactly 2k players have been assigned. In particular, after the m-th stage, all 2m players have been assigned.

At the beginning of a stage, suppose the remaining players are
\[
b_1\le b_2\le\cdots\le b_{2k},
\]
where $k \leq m. $

{\bf Case 1.} Assume first that there exists an index $r$ satisfying
\[
\frac{b_r+b_{r+1}}2
\le0 <
\frac{b_{r+1}+b_{r+2}}2.
\]

Alice first picks $b_{r+1}$.

Bob now has four possible responses.

\begin{enumerate}
\item Bob assigns $b_{r+1}$ into his team and picks $b_t$ where $b_t \geq b_{r+1}$. Then Alice assigns $b_t$ to her team.

\item Bob assigns $b_{r+1}$ into his team and picks $b_t$ where $b_t < b_{r+1}$. Then Alice assigns $b_t$ to Bob's team. 

\item Bob assigns $b_{r+1}$ into Alice's team and picks $b_t$ where $b_t \geq b_{r+1}$. Then Alice assigns $b_t$ in her team. 

\item Bob assigns $b_{r+1}$ into Alice's team and picks $b_t$ where $b_t < b_{r+1}$. Then Alice assigns $b_t$ in Bob's team. 
\end{enumerate}

{\bf Case 2.}
If no such index $r$ exists, then all consecutive pair averages are of the same sign. 
If they are all negative, Alice picks the largest remaining player. Independent of whatever Bob does, Alice assigns Bob's chosen player to Bob's team. 

Similarly, if they are all nonnegative, Alice picks the smallest remaining player. Again, independent of whatever Bob does, Alice assigns Bob's chosen player to Alice's team. 

From the strategy itself, it is clear that Alice's team will never be empty. At each stage, exactly one of the following five situations occurs.

\begin{enumerate}
\item Alice receives both players. Observe that if it happens then the average of these two players is non-negative. 

\item Bob receives both players. Observe that if it happens then the average of these two players is less than $0$. 

\item Both Alice and Bob receives one player, with Alice receiving a nonnegative player and Bob receiving a negative player.

\item Each player receives one negative player, with Alice's player  receiving the larger one (possibly equal).

\item Each player receives one positive player, with Alice receiving the larger one (possibly equal).
\end{enumerate}

After removing the two assigned players from further consideration, Alice repeats the same procedure until no players remain.
We now analyse the outcome.
Let

\begin{itemize}
\item $l_1$ be the number of stages where Alice receives both players;
\item $l_2$ be the number of stages where Bob receives both players;
\item $l_3$ be the number of stages where Alice receives a nonnegative player and Bob receives a negative player;
\item $l_4$ be the number of stages where Alice picks one, Bob picks one and both players are negative;
\item $l_5$ be the number of stages where Alice picks one, Bob picks one and both players are nonnegative;
\end{itemize}

Moreover, let

\begin{itemize}
\item $t_1$ denote the average, over the $l_1$ stages, of the sum of the strengths of the two players assigned to Alice;
\item $s_2$ denote the average, over the $l_2$ stages, of the sum of the strengths of the two players assigned to Bob;
\item $t_3$ and $s_3$ denote the average strengths of the players assigned to Alice and Bob, respectively, over the $l_3$ stages;
\item Let $t_4$, $t'_4$ be such that 
Alice's and Bob's average strengths over the 
$l_4$ stages are 
$-t_4$ and $-t_4'$ respectively;
\item $t_5$ and $t_5'$ denote the average strengths of the players assigned to Alice and Bob, respectively, over the $l_5$ stages.
\end{itemize}

By construction, we have 
\begin{enumerate}
    \item $t_1\ge0$,
    \item
 $ s_2 < 0,$
\item $t_3 > s_3$, 
\item $t_4 \leq t_4'$,
\item $t_5 \geq t_5'$.
\end{enumerate}

Therefore Alice's total strength equals

\[
2l_1t_1+l_3t_3+l_5t_5-l_4t_4,
\]

whereas Bob's total strength equals

\[
2l_2s_2+l_3s_3+l_5t_5'-l_4t_4'.
\]

Suppose Alice's total strength is negative. Then

\[
2l_1t_1+l_3t_3+l_5t_5-l_4t_4<0.
\]

Since
\[ 
2l_2s_2\le2l_1t_1, 
l_3s_3\le l_3t_3,
l_5t_5'\le l_5t_5,
l_4t_4'\ge l_4t_4,
\]

it follows that

\[
2l_2s_2+l_3s_3+l_5t_5'-l_4t_4'<0.
\]

Thus Bob's total strength is also negative.
This contradicts
the fact that the total strength of all players is zero.
Hence Alice's total strength is nonnegative, proving that her average strength is at least the overall average.

Finally, suppose the game ends in a draw, that is, equality holds throughout the above argument.

Then every inequality used above must in fact be an equality. Consequently, every stage of Types~$3$, $4$, and $5$ consists of two equal numbers. Indeed, $l_3=0$, while in every stage of Type~$4$ the two negative numbers are equal, and in every stage of Type~$5$ the two nonnegative numbers are equal. Hence every value arising from a stage of Types~$3$, $4$, or~$5$ occurs with even multiplicity.

Therefore, any value occurring with odd multiplicity must arise from a stage of Type~$1$ or~$2$, in which one player receives both numbers. Since the total contribution of these stages is zero, every such value must be accompanied by its negative with the same odd multiplicity. Thus, the multiset of values occurring with odd multiplicity is symmetric about the origin. Recalling that we translated the strengths so that the overall average is $0$, it follows that, in the original setting, this multiset is symmetric about~$x$.

Therefore, if the multiset of values occurring with odd multiplicity is not symmetric about $x$, then a draw is impossible. Since Alice can never lose, she must have a winning strategy. This completes the proof for even $n$. 

% Finally, suppose the game becomes a fraw, that is, assume equality holds.

% Then every inequality used above must itself be an equality. Consequently every stage of Types $3$, $4$, and $5$ consists of two equal numbers. Indeed, $l_e=0$, and in Type $4$ the 
% two negative numbers are equal; and in Type $5$ the two positive numbers are equal.
% Hence every value arising from these stages must have even multiplicity.

% Therefore every value occurring an odd number of times must arise from the stages where one player receives both numbers. Since these stages contribute total sum zero, every such value
% must be accompanied by its negative with the same odd multiplicity. Thus the multiset of values occurring an odd number of times is symmetric about the origin. Thus, the multiset of values occurring an odd number of times is symmetric about $x$. 

% Therefore, if the multiset of values occurring an odd number of times is not symmetric about $x$, Alice has a winning strategy. 
% This completes the proof for $n$ even. 

We now consider the case when $n$ is odd.

\medskip

\noindent
{\bf \large When $n$ is odd:}

\medskip

Let
\[
S=\sum_{i=1}^{2m+1}a_i=(2m+1)x.
\]

Suppose that, on the first move, Alice names the player of strength $a_i$.

If $a_i>x$, Bob assigns this player to his team. The remaining $2m$ players have average
\[
x'
=\frac{S-a_i}{2m}
=x+\frac{x-a_i}{2m}
<x.
\]
The remainder of the game is an even-player game with Bob choosing first. By the even-player case, Bob can ensure that his average strength over the remaining players is at least $x'$, and hence at least Alice's average strength over the remaining players. Since Bob's first player has strength $a_i>x'$, his overall average strength is strictly greater than $x'$. As $x'<x$, it follows that Bob's overall average strength is strictly greater than $x$.

Now suppose that $a_i<x$. Bob assigns this player to Alice's team. The remaining $2m$ players have average
\[
x'
=\frac{S-a_i}{2m}
=x+\frac{x-a_i}{2m}
>x.
\]
Again, the remainder of the game is an even-player game with Bob choosing first. By the even-player case, Bob can ensure that his average strength over the remaining players is at least $x'$. Since every player assigned to Bob's team comes from the remaining set, his final average strength is at least $x'>x$.

Finally, suppose that $a_i=x$. Bob assigns this player to his team. The remaining $2m$ players also have average $x$. If the multiset of values occurring with odd multiplicity is not symmetric about $x$, then, the remaining multiset also fails the symmetry condition. Hence the strict part of the even-player case applies, implying that Bob can ensure that his average strength over the remaining players is strictly greater than $x$. Since Bob's first player also has strength $x$, his overall average strength is likewise strictly greater than $x$.

This completes the proof.
\end{proof}

\begin{remark}
\label{rem:examples}
The condition that Alice's first move is $a_{r+1}$ is essential. Indeed, there are instances in which Bob has a winning strategy whenever Alice opens with any other player. For example, consider the game
\[
G(2,2,5,8,12,12).
\]
Here,
\[
x=\frac{41}{6},
\]
so $r=3$, and hence $a_{r+1}=8$. Therefore, Alice's optimal opening move is to choose the player of strength $8$. In fact, if Alice starts with any player other than $8$, Bob can force a win. More precisely:
\begin{itemize}
    \item if Alice chooses $5$, then Bob assigns $5$ to Alice's team and chooses $8$;
    \item if Alice chooses $2$, then Bob assigns $2$ to Alice's team and chooses $8$;
    \item if Alice chooses $12$, then Bob assigns $12$ to his own team and chooses $5$.
\end{itemize}
In each of the above cases, it is straightforward to verify that Bob wins under optimal play.
\end{remark}

\section*{Acknowledgements} 
The author thanks Bhargav Narayanan for writing such a beautiful paper and thus inspiring him to think hard about silly things. The author thanks Digjoy Paul for letting him know regarding this paper. The author thanks Dipnit Biswas, Indrajit Saha and Nilanjan Bag for discussions. 
 The author also acknowledges the INSPIRE Faculty Fellowship (Reference No. DST/INSPIRE/04/2024/004712; Faculty Registration No. IFA24-MA 205) for support 
during the preparation of this work, and thanks the Department of Science and Technology (DST), 
India, for funding. The author also appreciates the excellent research environment
provided by the Department of Mathematics at the Indian Institute of Technology Jammu.

\end{document}